\documentclass[paper=a4,english,fontsize=11pt,parskip=half,abstract=true]{scrartcl}
\usepackage{babel}
\usepackage[utf8]{inputenc}
\usepackage[T1]{fontenc}
\usepackage[left=20mm,right=20mm,top=30mm,bottom=30mm]{geometry}
\usepackage{amsmath}
\usepackage{amsthm}
\usepackage{amssymb}
\usepackage{enumerate} 
\usepackage{thmtools}
\usepackage{mathtools}
\mathtoolsset{centercolon} 
\usepackage[bookmarks=true,
            pdftitle={Real blocks with dihedral defect groups revisited},
            pdfauthor={Benjamin Sambale},
            pdfkeywords={},
            pdfstartview={FitH}]{hyperref}

\newtheorem{Thm}{Theorem} 
\newtheorem{Lem}[Thm]{Lemma}
\newtheorem{Prop}[Thm]{Proposition}

\newtheorem{Con}[Thm]{Conjecture}
\theoremstyle{definition}

\numberwithin{equation}{section}

\allowdisplaybreaks[1]
\newenvironment{sm}{\bigl(\begin{smallmatrix}}{\end{smallmatrix}\bigr)}

\renewcommand{\phi}{\varphi}
\newcommand{\C}{\mathrm{C}}
\newcommand{\N}{\mathrm{N}}
\newcommand{\Z}{\mathrm{Z}}
\newcommand{\pcore}{\mathrm{O}}

\newcommand{\ZZ}{\mathbb{Z}}

\newcommand{\QQ}{\mathbb{Q}}

\newcommand{\FF}{\mathbb{F}}

\newcommand{\Aut}{\operatorname{Aut}}

\newcommand{\GL}{\operatorname{GL}}
\newcommand{\SL}{\operatorname{SL}}

\newcommand{\PSL}{\operatorname{PSL}}
\newcommand{\PGL}{\operatorname{PGL}}

\newcommand{\Irr}{\operatorname{Irr}}
\newcommand{\IBr}{\operatorname{IBr}}

\newcommand{\CF}{\mathcal{F}}

\newcommand{\TT}{\mathrm{t}}

\title{Real blocks with dihedral defect groups revisited}
\author{Benjamin Sambale\footnote{Institut für Algebra, Zahlentheorie und Diskrete Mathematik, Leibniz Universität Hannover, Welfengarten 1, 30167 Hannover, Germany,
\href{mailto:sambale@math.uni-hannover.de}{sambale@math.uni-hannover.de}}}
\date{\today}

\begin{document}
\frenchspacing
\maketitle
\begin{abstract}\noindent
The Frobenius--Schur indicators of characters in a real $2$-block with dihedral defect groups have been determined by Murray. We show that two infinite families described in his work do not exist and we construct examples for the remaining families. We further present some partial results on Frobenius--Schur indicators of characters in other tame blocks.
\end{abstract}


\renewcommand{\sectionautorefname}{Section}
\section{Introduction}

Finite groups with dihedral Sylow $2$-subgroups were fully classified by Gorenstein--Walter \cite{GW0,GW,GW2,GW3} (an alternate proof was given by Bender~\cite{BenderD,BenderGlauberman}). The principal $2$-blocks of such groups were investigated by Brauer~\cite{BrauerApp3}, Erdmann~\cite{ErdmannDprincipal},  Landrock~\cite{LandrockDprincipal} and recently Koshitani--Lassueur~\cite{KoshitaniL}. As a natural next step, it is desirable to understand arbitrary blocks $B$ of finite groups $G$ with dihedral defect groups $D$ of order $2^d\ge 4$. Brauer~\cite{Brauer} has shown that the number of irreducible characters in $B$ is $k(B)=2^{d-2}+3$, where four of them have height $0$ and the remaining characters have height $1$. On the other hand, the number of simple modules in $B$ is $l(B)=1$, $2$ or $3$ depending on three different fusion patterns (if $B$ is the principal block, the fusion patterns are distinguished by the number of conjugacy classes of involutions: there are three, two or only one such class respectively). Based on Brauer's computations, Cabanes--Picaronny~\cite{Cabanes} have constructed perfect isometries between blocks with dihedral defect groups and the same fusion pattern. 

The algebra structure of $B$ was first investigated for solvable groups $G$ by Erdmann--Michler~\cite{ErdmannMichler} and Koshitani~\cite{KoshitaniDsolvable}. The general case of arbitrary groups was subsequently studied by Donovan~\cite{DonovanD} and by Erdmann~\cite{ErdmannD,Erdmann} in the framework of tame algebras. Some of the algebras with two simple modules described by Erdmann were not known to occur as block algebras. For $l(B)=3$, Linckelmann~\cite{LinckelmannDerived} has lifted the perfect isometries constructed by Cabanes--Picaronny to derived equivalences. This applies in particular to Klein four defect groups (i.\,e. $d=2$) were one has stronger results by Linckelmann~\cite{LinckelmannC2C2} and Craven--Eaton--Kessar--Linckelmann~\cite{CEKL} on the source algebra of $B$ (the case $l(B)=2$ does not occur here). The derived equivalence classes for $l(B)=2$ were later found by Holm~\cite{HolmPub}.
The possible Morita equivalence classes in this situation were restricted by Bleher~\cite{Bleher0,Bleher} and Bleher--Llosent--Schaefer~\cite{Bleher2} using universal deformation rings. Thereafter, Eisele~\cite{EiseleD} proved that certain scalars in Erdmann's description of the basic algebra cannot arise for $l(B)=2$.
Finally, using the classification of finite simple groups, a complete list of all Morita equivalence classes of blocks with dihedral defect groups was given recently by Macgregor~\cite{Macgregor}.

Some questions on blocks cannot even be answered when the Morita equivalence class is known. For instance, if $B$ contains real characters $\chi\in\Irr(B)$, it is of interest to determine their Frobenius--Schur indicators (F-S indicators for short) 
\[\epsilon(\chi):=\frac{1}{|G|}\sum_{g\in G}\chi(g^2)\]
in terms of $D$. Since $\chi(g^2)$ can only be non-zero when the square of the $2$-part of $g$ is conjugate to an element in $D$, it is plausible that $\epsilon(\chi)$ actually depends on an extension $E$ of $D$ such that $|E:D|=2$. 
In fact, Murray~\cite{MurraySubpairs} has described the F-S indicators when $D$ is a dihedral group using the decomposition matrix and the so-called extended defect group of $B$. It is however not clear which combinations of these ingredients can actually occur. 
The aim of this note is to eliminate two infinite families of Murray's classification and construct explicit examples for the remaining cases.

\begin{Thm}\label{main}
Let $B$ be real block of a finite group $G$ with dihedral defect group $D$ of order $2^d\ge 8$ and extended defect group $E$. Let $\epsilon_1,\ldots,\epsilon_4$ be the F-S indicators of the four irreducible characters of height $0$ in $B$. 
There is a unique family of $2$-conjugate characters of height $1$ in $\Irr(B)$ of size $2^{d-3}$. Let $\mu$ be the common F-S indicator of those characters. The possible values for $\epsilon_1,\ldots,\epsilon_4,\mu$ are given as follows, while the remaining $2^{d-3}-1$ characters \textup(of height $1$\textup) all have F-S indicator $1$:
\begin{center}
\begin{tabular}{llll}
Morita equivalence class&$l(B)$&$E$&$\epsilon_1,\ldots,\epsilon_4;\mu$\\\hline
$D$ \textup(nilpotent\textup)&$1$&$D$, $D\times C_2$&$1,1,1,1;1$\\
&&$D*C_4$&$1,1,1,1;-1$\\
&&$D_{2^{d+1}}$&$0,0,1,1;1$\\
&&$SD_{2^{d+1}}$&$0,0,1,1;-1$\\
&&$C_{2^{d-1}}\rtimes C_2^2$, $d\ge 4$&$1,1,1,1;0$\\\hline
$\PGL(2,q)$, $|q-1|_2=2^{d-1}$&$2$&$D$, $D\times C_2$&$1,1,1,1;1$\\
&&$C_{2^{d-1}}\rtimes C_2^2$, $d\ge 4$&$1,1,1,1;0$\\\hline
$\PGL(2,q)$, $|q+1|_2=2^{d-1}$&$2$&$D$, $D\times C_2$&$1,1,1,1;1$\\\hline

$\PSL(2,q)$, $|q-1|_2=2^d$&$3$&$D$, $D\times C_2$&$1,1,1,1;1$\\
&&$D_{2^{d+1}}$&$0,0,1,1;1$\\
&&$SD_{2^{d+1}}$&$0,0,1,1;-1$\\
&&$C_{2^{d-1}}\rtimes C_2^2$, $d\ge 4$&$1,1,1,1;0$\\\hline
$\PSL(2,q)$, $|q+1|_2=2^d$&$3$&$D$, $D\times C_2$&$0,0,1,1;1$\\
&&$D_{2^{d+1}}$&$1,1,1,1;1$\\\hline
$A_7$, $d=3$&$3$&$D$, $D\times C_2$&$1,1,1,1;1$
\end{tabular}
\end{center}
All cases occur for all $d$ as indicated.
\end{Thm}
The proof of \autoref{main} is given in the next section. 
In the third section we refine one of the conjectures made in \cite{SambaleReal}. In this context we present to two new general results in the subsequent section. We apply these results in \autoref{secQ} to obtain partial information on the F-S indicators of characters in arbitrary tame blocks. Finally, we determine all F-S indicators in real blocks with homocyclic defect group of type $C_4\times C_4$.  

\section{Proof of \autoref{main}}

Our notation is fairly standard and follows \cite{SambaleReal}. We assume a very basic understanding of fusion systems and refer to \cite{LinckelmannBook2} occasionally. 
In the following let $B$ be a $2$-block of a finite group $G$. We may assume that $B$ is real, i.\,e. $\Irr(B)$ is invariant under complex conjugation (otherwise all characters in $\Irr(B)$ have F-S indicator $0$). Recall that $B$ determines up to conjugation a unique \emph{defect pair} $(D,E)$ such that $D$ is a defect group and $E$ is an extended defect group of $B$ (see \cite[Section~3]{SambaleReal} for details). 
We remark that $D\le E$ and $|E:D|\le2$ with equality if and only if $B$ is not the principal block. 

Now let $D$ be a dihedral group of order $2^d\ge 4$. Then $B$ is nilpotent if and only if $l(B)=1$. For $l(B)>1$ and $d\ge 3$, we have observed that two type (b) cases in \cite[Table~2]{MurraySubpairs} have no counterparts in \cite[Theorems~1.7, 1.8]{MurrayCyclic2} for $d=2$ (i.\,e. $D$ is a Klein four-group). In fact, the following proposition shows that these two cases in \cite[Table~2]{MurraySubpairs} do not occur.

\begin{Prop}\label{nonexist}
Let $B$ be a real $2$-block of a finite group $G$ with defect pair $(D,E)$ such that $D\cong D_{2^d}$ with $d\ge 3$. If $l(B)>1$, then $E\cong D\times C_2$ or $\C_E(D)=\Z(D)$. 
\end{Prop}
\begin{proof}
Let $b_D$ be a Brauer correspondent of $B$ in $D\C_G(D)$. Then by \cite[Lemma~2.2]{MurraySubpairs}, we may choose $E$ in such a way that $b_D^{E\C_G(D)}$ is real with defect pair $(D,E)$. In other words, $(D,b_D,E)$ is a Sylow $B$-subtriple in the notation of \cite{MurraySubpairs}. Since $B$ is not nilpotent, there exists a so-called \emph{essential} subgroup $Q\le D$ in the fusion system $\CF$ of $B$. Then $Q$ is a Klein four-group and there exists a unique $B$-subpair $(Q,b_Q)\le(D,b_D)$ (see \cite[Theorem~1]{Sambale}). Moreover, $b_Q$ is nilpotent with defect group $\C_D(Q)=Q$ (see \cite[Theorem~IV.3.19]{AKO}). Since $Q$ is essential, 
\[\N_G(Q,b_D)/\C_G(Q)\cong \Aut_\CF(Q)\cong S_3.\] 
The block $B_Q:=b_Q^{\N_G(Q,b_Q)}$ has defect group $\N_D(Q)\cong D_8$ by \cite[Theorem~IV.3.19]{AKO}. By \cite[Corollary~9.21]{Navarro}, $B_Q$ is the only block of $\N_G(Q,b_Q)$ that covers $b_Q$. 
Since every subgroup of $S_3$ has trivial Schur multiplier, each $\psi\in\Irr(b_Q)$ extends to its inertial group. The number of extensions is determined by Gallagher's theorem. If $\psi$ is $\N_G(Q,b_Q)$-invariant, then $\Irr(B_Q)$ contains three extensions of $\psi$. Since $k(B_Q)=5$, there can be at most one such character $\psi$. If $\N_G(Q,b_Q)$ has two orbits of length $2$ in $\Irr(b_Q)$, then we would get six characters in $\Irr(B_Q)$. Therefore, the four characters in $\Irr(b_Q)$ distribute into orbits of length $1$ and $3$ under the action of $\N_G(Q,b_Q)$. In particular, $\Irr(b_Q)$ contains (at least) three characters with the same F-S indicator. 

The possible extended defect groups $E$ were determined in \cite[Proposition~4.1]{MurraySubpairs}. Suppose that our claim is false. Then we are in case (b), i.\,e. $E\cong D*C_4$ is a central product. By \cite[Lemma~2.6]{MurraySubpairs}, $b_Q$ is real and has extended defect group $\C_E(Q)=Q*E\cong C_4\times C_2$. But now, \cite[Theorem~1.7]{MurrayCyclic2} implies that exactly two characters in $\Irr(b_Q)$ have F-S indicator $1$. This contradicts the observation above.
\end{proof}

In order to show that the remaining cases in \cite[Table~2]{MurraySubpairs} occur, we provide a general construction. 

\begin{Prop}\label{examples}
Let $H<\hat{H}$ be finite groups such that $|\hat{H}:H|=2$. Let $E$ be a Sylow $2$-subgroup of $\hat{H}$ and let $D:=E\cap H$. Let $H\times C_3<G<\hat{H}\times S_3$ such that $H\times S_3\ne G\ne\hat{H}\times C_3$. Then $G$ has a real $2$-block $B$ with defect pair isomorphic to $(D,E)$. Moreover, $B$ is Morita equivalent to the principal block $B_0(H)$ of $H$ and $B$ and $B_0(H)$ have the same fusion system.
\end{Prop}
\begin{proof}
Note that $B_0(H)$ is isomorphic to a (non-real) block $B_0(H)\otimes b$ of $H\times C_3$, where $b\cong F$ is a non-principal block of $C_3$.
Clearly $B_0(H)$ and $B_0(H)\otimes b$ have the same fusion system. 
Let 
\[B:=(B_0(H)\otimes b)^G\]
be the Fong--Reynolds correspondent of $B_0(H)\otimes b$ in $G$. By \cite[Theorem~6.8.3]{LinckelmannBook2}, $B$ is Puig equivalent to $B_0(H)\otimes b$ (i.\,e. the blocks have the same source algebra). This implies that $B$ is Morita equivalent to $B_0(H)$ and both blocks have the same fusion system (see \cite[Theorem~8.7.1]{LinckelmannBook2}).
In particular, $B$ has defect group $D$. 

Let $\Irr(b)=\{\theta\}$. Then $\chi:=(1_H\times\theta)^G\in\Irr(B)$ is a real character and therefore $B$ is real. Moreover, $B$ is not the principal block of $G$ since otherwise $B$ cannot cover $B_0(H)\otimes b$. Therefore, an extended defect group of $B$ must be a Sylow $2$-subgroup of $G$, because $|G:H\times C_3|=2$. Let $e\in E\setminus D$ and $x\in S_3$ an involution. Then $D\langle ex\rangle$ is a Sylow $2$-subgroup of $G$ and 
\[E\to D\langle ex\rangle,\qquad g\mapsto\begin{cases}
g&\text{if }g\in D,\\
gx&\text{if }g\notin D
\end{cases} \]
is an isomorphism.
\end{proof}

Choosing $(H,\hat{H})=(D,E)$ in \autoref{examples} shows that there are nilpotent real blocks for every given defect pair $(D,E)$.
Similarly, the choice $\hat{H}=H\times C_2$ leads to $G=H\times S_3$ and a block with extended defect group $E\cong D\times C_2$.

\begin{proof}[Proof of \autoref{main}]
By \autoref{nonexist} and \cite[Table~2]{MurraySubpairs}, it remains to construct examples for each Morita equivalence class and each defect pair. By the remark above, we may assume that $B$ is not nilpotent. 
Then by \cite[Theorem~2.1]{Macgregor}, $B$ is Morita equivalent to $B_0(H)$ where $H$ is one of the following groups
\begin{enumerate}[(1)]
\item\label{PGL1} $\PGL(2,q)$ with $|q-1|_2=2^{d-1}$.
\item\label{PGL3} $\PGL(2,q)$ with $|q+1|_2=2^{d-1}$. 
\item\label{PSL1} $\PSL(2,q)$ with $|q-1|_2=2^d$.
\item\label{PSL3} $\PSL(2,q)$ with $|q+1|_2=2^d$. 
\item\label{A7} $A_7$ with $d=3$.
\end{enumerate}
Note that for every $d\ge 3$ there exists a prime $q$ congruent to $\pm1+2^d$ modulo $2^{d+1}$ by Dirichlet's theorem. Thus, appropriate groups $H$ exist for every $d$. Moreover, if $q\equiv 1+2^d\pmod{2^{d+1}}$, then $q^2\equiv 1+2^{d+1}\pmod{2^{d+2}}$. This will be used later on.

For the cases \eqref{PGL3} and \eqref{A7}, Murray~\cite[Table~2]{MurraySubpairs} has shown that only $E\cong D\times C_2$ is possible. Here we take $G=H\times S_3$ as explained above.
In case \eqref{PSL3}, we find $E\cong D_{2^{d+1}}$ in \cite[Table~2]{MurraySubpairs}. Since \[H\cong\SL(2,q)\Z(\GL(2,q))/\Z(\GL(2,q))\le\PGL(2,q),\] 
we can take $\hat{H}:=\PGL(2,q)$ with the required properties. 

Suppose now that case \eqref{PGL1} occurs.
By the remark above we find a prime $p$ such that $q=p^2\equiv 1+2^d\pmod{2^{d+1}}$. Let $\sigma$ be the Frobenius automorphism $\FF_q\to\FF_q$, $x\mapsto x^p$. Then the semilinear group $\hat{H}:=H\rtimes\langle \sigma\rangle$ has Sylow $2$-subgroup $E\cong C_{2^{d-1}}\rtimes C_2^2$ where $C_2^2$ acts faithfully on $C_{2^{d-1}}$ (this is type (e) in \cite[Proposition~4.1]{MurraySubpairs}).  

Next, consider case \eqref{PSL1}. The choice $\hat{H}:=\PGL(2,q)$ realizes $E\cong D_{2^{d+1}}$. 
We may therefore assume that $q=p^2$ as above. Then there exists a subgroup $\hat{H}$ of $\PGL(2,q)\rtimes\langle \sigma\rangle$ of index $2$ with semidihedral Sylow $2$-subgroup $E\cong SD_{2^{d+1}}$ (this group is denoted by $\PGL(2,q)^*$ in \cite{GorensteinPGLstar}).
Finally, let $d\ge 4$. Here $\hat{H}:=H\rtimes\langle \sigma\rangle$ has Sylow $2$-subgroup $E\cong C_{2^{d-1}}\rtimes C_2^2$ as above. 
\end{proof}

\section{A refined conjecture}

In \cite[Conjecture~C]{SambaleReal}, the following conjecture was put forward.

\begin{Con}\label{conC}
Let $B$ be a real, non-principal $2$-block with defect pair $(D,E)$ and a unique projective indecomposable character $\Phi$. 
Then 
\[\epsilon(\Phi)=|\{x\in E\setminus D:x^2=1\}|.\]
\end{Con}

Let $\Phi_\phi$ be the projective indecomposable character attached to some $\phi\in\IBr(B)$. Murray~\cite[Lemma~2.6]{MurrayCyclic2} has shown that $\epsilon(\Phi_\phi)$ is the multiplicity of $\phi$ as a constituent of the permutation character on $\{x\in G:x^2=1\}$ (see \autoref{lemPhix} for a refinement). 
I now believe that the conjecture holds orbit-by-orbit as follows.

\begin{Con}\label{conNew}
Let $B$ be a real, non-principal $2$-block with defect pair $(D,E)$ and a unique projective indecomposable character $\Phi$. 
Then for every involution $x\in G$, 
\[[\Phi_{\C_G(x)},1_{\C_G(x)}]=|x^G\cap E\setminus D|,\]
where $x^G$ denotes the conjugacy class of $x$ in $G$.
\end{Con}

Note that $[\Phi_{\C_G(x)},1_{\C_G(x)}]=|D|[\phi_{\C_G(x)},1_{\C_G(x)}]^0$ where $\IBr(B)=\{\phi\}$ in the situation of \autoref{conNew}.

\begin{Thm}
If \autoref{conNew} holds for $B$, then \autoref{conC} holds for $B$.
\end{Thm}
\begin{proof}
Let $\Omega:=\{x\in G:x^2=1\}$. Note that the equation in \autoref{conNew} is also true for $x=1$ since $B$ is not principal. Recall that $\Phi$ vanishes on the elements of even order. By the definition of F-S indicators and \autoref{conNew}, we have 
\begin{align*}
\epsilon(\Phi)&=\frac{1}{|G|}\sum_{g\in G}\Phi(g^2)=\frac{1}{|G|}\sum_{x\in\Omega}\sum_{h\in\C_G(x)^0}\Phi(h^2)=\sum_{x\in\Omega/G}[\Phi_{\C_G(x)},1_{\C_G(x)}]\\
&=\sum_{x\in\Omega/G}|x^G\cap E\setminus D|=|\Omega\cap E\setminus D|=|\{x\in E\setminus D:x^2=1\}|.\qedhere
\end{align*}
\end{proof}

\section{Two general lemmas}

In this section we prove two new results, which are related to \autoref{conC}. These will be applied in the subsequent sections.

Recall that a $B$-subsection is a pair $(x,b_x)$ where $x\in D$ and $b_x$ is a Brauer correspondent of $B$ in $\C_G(x)$. 
In \cite[remark before Theorem~13]{SambaleReal}, we explained that, after conjugation, we may assume that $b_x$ has defect pair $(\C_D(x),\C_E(x))$ (if $b_x$ is non-real, then $\C_D(x)=\C_E(x)$). 
For $\chi\in\Irr(B)$ and $\phi\in\IBr(b_x)$ we denote the corresponding generalized decomposition number by $d_{\chi\phi}^x$. 
In analogy to principal indecomposable modules, we set $\Phi_\phi^x:=\sum_{\chi\in\Irr(B)}d_{\chi\phi}^x\chi$. 

The following result generalizes \cite[Lemma~2.6]{MurrayCyclic2}.

\begin{Lem}\label{lemPhix}
Let $B$ be a real $2$-block with defect pair $(D,E)$ and subsection $(x,b)$. Let $\pi$ be the Brauer permutation character of the conjugation action of  $\C_G(x)$ on $\Omega_x:=\{y\in G:y^2=x\}$. Then the multiplicity of $\phi\in\IBr(b)$ as a constituent of $\pi$ is
\[\epsilon(\Phi_\phi^x)=\sum_{\chi\in\Irr(B)}\epsilon(\chi)d_{\chi\phi}^x.\]
In particular, $\epsilon(\Phi_\phi^x)$ is a non-negative integer. If there is no $e\in E\setminus D$ such that $e^2=x$, then $\epsilon(\Phi_\phi^x)=0$. 
\end{Lem}
\begin{proof}
By Brauer's formula~\cite[Theorem~4A]{BrauerApp3} (see also \cite[Lemma~2]{SambaleReal}) we have
\[\epsilon(\Phi_\phi^x)=\sum_{\psi\in\Irr(b)}\epsilon(\psi)d^x_{\psi\phi}.\]
Since $\Omega_x=\{y\in\C_G(x):y^2=x\}$, we may assume that $x\in\Z(G)$ and $B=b$. By Brauer's second main theorem, the claim only depends on $\phi$, but not on $B$. 
For $g\in G^0$ we compute
\begin{align*}
\sum_{\phi\in\IBr(G)}\epsilon(\Phi_\phi^x)\phi(g)&=\sum_{\chi\in\Irr(G)}\epsilon(\chi)\sum_{\phi\in\IBr(G)}d_{\chi\phi}^x\phi(g)=\sum_{\chi\in\Irr(G)}\epsilon(\chi)\chi(xg)\\
&=|\{y\in G:y^2=xg\}|=|\{y\in\C_G(g):y^2=xg\}|
\end{align*}
(\cite[p. 49]{Isaacs}). Since $g$ has odd order, there exists a unique power $\sqrt{g}$ of $g$ such that $\sqrt{g}^2=g$. It is easy to check that the map $\{y\in\C_G(g):y^2=x\}\to\{y\in\C_G(g):y^2=xg\}$, $y\mapsto y\sqrt{g}$ is a bijection. Hence,
\[\sum_{\phi\in\IBr(G)}\epsilon(\Phi_\phi^x)\phi(g)=|\{y\in\C_G(g):y^2=x\}|=|\C_G(g)\cap\Omega_x|=\pi(g)\]
for all $g\in G^0$. Therefore, $\epsilon(\Phi_\phi^x)$ is the multiplicity of $\phi$ in $\pi$.

Now assume that is no $e\in E\setminus D$ such that $e^2=x$. 
Then \cite[Lemma~1.3]{MurraySubpairs} implies
\[0=\sum_{\chi\in\Irr(B)}\epsilon(\chi)\chi(x)=\sum_{\chi\in\Irr(B)}\epsilon(\chi)\sum_{\phi\in\IBr(b)}d_{\chi\phi}^x\phi(1)=\sum_{\phi\in\IBr(b)}\epsilon(\Phi_\phi^x)\phi(1),\]
and the second claim follows. 
\end{proof}

Our second lemma generalizes \cite[Theorem~10]{SambaleReal}.  

\begin{Prop}\label{locnil}
Let $B$ be a real, non-principal $2$-block with defect pair $(D,E)$. Let $(x,b)$ be a $B$-subsection such that $b$ is nilpotent with defect pair $(\C_D(x),\C_E(x))$ where $\C_D(x)$ is abelian. Then 
\[\epsilon(\Phi_\phi^x)=|\{e\in E\setminus D:e^2=x\}|,\]
where $\IBr(b)=\{\phi\}$.
\end{Prop}
\begin{proof}
As in the proof of \autoref{lemPhix} we may apply Brauer's formula. 
Since $b$ has defect pair $(\C_D(x),\C_E(x))$ and 
\[\{e\in E\setminus D:e^2=x\}=\{e\in\C_E(x)\setminus\C_D(x):e^2=x\},\] 
we may assume that $x\in\Z(G)$ and $B=b$. Now $D$ is abelian and \autoref{conC} holds for $B$ by \cite[Theorem~10]{SambaleReal}.
Since every Brauer correspondent $\beta$ of $B$ in a section of $G$ is nilpotent with abelian defect groups, \autoref{conC} also holds for $\beta$. The claim follows from \cite[Theorem~13]{SambaleReal}.
\end{proof}

In the situation of \autoref{locnil} it is tempting to formalize a local version of \autoref{conNew}: For every $y\in G$ with $y^2=x$, we have
\[[\Phi_{\C_G(y)},1_{\C_G(y)}]=|y^{\C_G(x)}\cap E\setminus D|\]
(note that $\C_G(y)\subseteq\C_G(x)$).
We did not find any counterexamples to this equation.

\section{Tame blocks}\label{secQ}

By Murray~\cite[Theorems~1.7 and 1.8]{MurrayCyclic2}, the F-S indicators of blocks with Klein four defect group are known. As in the proof of \autoref{main} one can show that all cases listed there occur. 
It is tempting to do a similar analysis for other tame blocks, i.\,e. $2$-blocks with quaternion or semidihedral defect groups. 
In this section we gather some partial results along these lines.

\begin{Prop}\label{tameh0}
Let $B$ be a real tame block of a finite group with defect at least $3$. Then $B$ has two or four real irreducible characters of height $0$ and they all have F-S indicator $1$.
\end{Prop}
\begin{proof}
Let $(D,E)$ be a defect pair of $B$. It is well-known that $B$ has exactly four irreducible characters of height $0$ (see \cite[Theorem~8.1]{habil}). By \cite[Theorem~5.1]{Gow}, at least one such character has F-S indicator $1$. Since non-real characters come in pairs of the same degree, $B$ has two or four real irreducible characters of height $0$. 

To prove the second claim, it suffices to show that $D/D'$ has a complement in $E/D'$ by \cite[Theorem~5.6]{Gow}. 
Since $D/D'\cong C_2^2$, we may assume that $E/D'\cong C_4\times C_2$ by way of contradiction. In particular, $|E:E'|\ge 8$. A theorem of Alperin--Feit--Thompson asserts that the number of involutions in $E$ is congruent to $3$ modulo $4$ (see \cite[Theorem~4.9]{Isaacs}). By the remark after \cite[Theorem~4.9]{Isaacs}, the number of involutions in $D$ is congruent to $1$ modulo $4$. Hence, there exists an involution $x\in E\setminus D$. But then $\langle x\rangle D'/D'$ is a complement of $D/D'$ in $E/D'$. Contradiction.
\end{proof}

A non-principal block of $G=(C_3\rtimes C_4)\times C_2$ shows that \autoref{tameh0} fails for tame blocks of defect $2$ (i.\,e. blocks with Klein four defect group). 

We now apply \autoref{locnil} to a concrete example.

\begin{Prop}\label{Q8h0}
Let $B$ be a block with defect pair $(D,E)$ where $D\cong Q_8$. Then $B$ has exactly two real irreducible characters of height $0$ if and only if one of the following holds
\begin{enumerate}[(1)]
\item $l(B)=1$ and $E\in\{Q_{16},SD_{16}\}$.
\item $B$ is Morita to the principal block of $\SL(2,3)$ and $E\notin\{Q_{16},SD_{16}\}$.
\item $B$ is Morita to the principal block of $\SL(2,5)$ and $E\in\{Q_{16},SD_{16}\}$.
\end{enumerate}
\end{Prop}
\begin{proof}
Let $\epsilon_1,\ldots,\epsilon_4$ be the F-S indicators of the height $0$ characters $\lambda_1,\ldots,\lambda_4\in\Irr(B)$. 
We may choose a $B$-subsection $(x,b)$ such that $|\langle x\rangle|=4$ and $b$ has defect pair $(\langle x\rangle,\C_E(x))$. 
Clearly, $b$ is nilpotent with abelian defect group $\langle x\rangle$. Let $\IBr(b)=\{\phi_x\}$. By the orthogonality relations of generalized decomposition numbers (see \cite[Theorem~1.14]{habil}), we have $d_{\lambda_i,\phi_x}^x=\pm1$ for $1\le i\le 4$ and $d_{\chi,\phi_x}^x=0$ for $\chi\in\Irr(B)\setminus\{\lambda_1,\ldots,\lambda_4\}$. These numbers depend on the ordinary decomposition matrix of $B$. 

\textbf{Case~1:} $l(B)=1$.\\
Here $B$ is nilpotent with decomposition matrix $(1,1,1,1,2)^\TT$. We may choose our labeling in such a way that $(d_{\lambda_1,\phi_x}^x,\ldots,d_{\lambda_4,\phi_x}^x)=(1,1,-1,-1)$. Similarly, there are elements $y,xy\in D$ of order $4$ such that
$(d_{\lambda_1,\phi_y}^y,\ldots,d_{\lambda_4,\phi_y}^y)=(1,-1,1,-1)$ and $(d_{\lambda_1,\phi_{xy}}^{xy},\ldots,d_{\lambda_4,\phi_{xy}}^{xy})=(1,-1,-1,1)$ with appropriate labeling. 
If there exists no $e\in E\setminus D$ such that $e^2\in\{x,y,xy\}$, then $(\epsilon_1,\ldots,\epsilon_4)=(1,1,1,1)$ by Propositions~\ref{locnil} and \ref{tameh0}. Now suppose that $e^2=x$ for some $e\in E\setminus D$. Then $e$ has order $8$ and it follows easily that $E\cong\{Q_{16},SD_{16}\}$. In both cases we have $|\{e\in E\setminus D:e^2=x\}|=2$ and $(\epsilon_1,\ldots,\epsilon_4)=(1,1,0,0)$ by Propositions~\ref{locnil} and \ref{tameh0}.

Now suppose that $l(B)>1$. By Macgregor~\cite[Corollary~2.4]{Macgregor}, there are two cases to consider.

\textbf{Case~2:} $B$ is Morita equivalent to the principal block of $\SL(2,3)$.\\
Here $l(B)=3$ and $B$ has decomposition matrix
\[
Q:=\begin{pmatrix}
1&.&.\\
.&1&.\\
.&.&1\\
1&1&1\\
1&1&.\\
1&.&1\\
.&1&1
\end{pmatrix}.
\]
We see that $\lambda_4$ is real and $\epsilon_4=1$. 
The orthogonality relations imply that $(d_{\lambda_1,\phi_x}^x,\ldots,d_{\lambda_4,\phi_x}^x)=\pm(1,1,1,-1)$. If there exists $e\in E\setminus D$ with $e^2=x$ (i.\,e. $E\in\{Q_{16}, SD_{16}\}$), then $(\epsilon_1,\ldots,\epsilon_4)=(1,1,1,1)$ and otherwise
$(\epsilon_1,\ldots,\epsilon_4)=(0,0,1,1)$ by Propositions~\ref{locnil} and \ref{tameh0} (after relabeling if necessary).

\textbf{Case~3:} $B$ is Morita equivalent to the principal block of $\SL(2,5)$.\\
Here $B$ has decomposition matrix
\[
Q:=\begin{pmatrix}
1&.&.\\
1&1&.\\
1&.&1\\
1&1&1\\
.&1&.\\
.&.&1\\
2&1&1
\end{pmatrix}.
\]
It follows that $\lambda_1$, $\lambda_4$ are real and $(d_{\lambda_1,\phi_x}^x,\ldots,d_{\lambda_4,\phi_x}^x)=\pm(1,-1,-1,1)$. If $E\in\{Q_{16},SD_{16}\}$, then $(\epsilon_1,\ldots,\epsilon_4)=(1,0,0,1)$ and $(\epsilon_1,\ldots,\epsilon_4)=(1,1,1,1)$ otherwise.
\end{proof}

To compute the remaining F-S indicators in the situation of \autoref{Q8h0} we restrict ourselves further to $E\in\{Q_{16},SD_{16}\}$.

\begin{Prop}
Let $B$ be a real block of a finite group $G$ with defect pair $(D,E)$ such that $D\cong Q_8$ and $E\cong\{Q_{16},SD_{16}\}$. Then the F-S indicators of characters in $\Irr(B)$ are given below, where the first four characters have height $0$. Moreover, all cases occur.
\begin{center}
\begin{tabular}{llll}
$E$&Morita equivalence class&$l(B)$&F-S indicators\\\hline
$Q_{16}$&$D$&$1$&$0,0,1,1;-1$\\
&$\SL(2,3)$&$3$&$1,1,1,1;-1,-1,-1$\\
&$\SL(2,5)$&$3$&$0,0,1,1;0,0,-1$\\\hline
$SD_{16}$&$D$&$1$&$0,0,1,1;1$\\
&$\SL(2,3)$&$3$&$1,1,1,1;1,1,1$\\
&$\SL(2,5)$&$3$&$0,0,1,1;0,0,1$
\end{tabular}
\end{center}
\end{Prop}
\begin{proof}
We reuse the notation from the proof of \autoref{Q8h0}. 

\textbf{Case~1:} $l(B)=1$.\\
By \autoref{Q8h0}, $(\epsilon_1,\ldots,\epsilon_4)=(1,1,0,0)$. Let $\Irr(B)=\{\lambda_1,\ldots,\lambda_4,\psi\}$ and $\IBr(B)=\{\phi\}$. 
Then 
\[2+2\mu=\epsilon_1d_{\lambda_1,\phi}+\ldots+\epsilon_4d_{\lambda_4,\phi}+\mu d_{\psi,\phi}\ge 0\]
with equality if and only if $E\cong Q_{16}$ by \cite[Lemma~1.3]{MurraySubpairs}. Hence, $\mu=-1$ if $E\cong Q_{16}$ and $\mu=1$ if $E\cong SD_{16}$. Examples for both cases can be constructed by the remark after \autoref{examples}. The groups are $\mathtt{SmallGroup}(48,18)$ for $E\cong Q_{16}$ and $\mathtt{SmallGroup}(48,17)$ for $E\cong SD_{16}$ in the small groups library~\cite{GAPnew}. 

Now let $l(B)=3$ and $\IBr(B)=\{\phi_1,\phi_2,\phi_3\}$. Let $\psi_1,\psi_2,\psi_3\in\Irr(B)$ be the characters of height $1$. Let $\mu_i:=\epsilon(\psi_i)$ for $i=1,2,3$.

\textbf{Case~2:} $B$ is Morita equivalent to the principal block of $\SL(2,3)$.\\
By \autoref{Q8h0}, $(\epsilon_1,\ldots,\epsilon_4)=(1,1,1,1)$. Assume first that $E\cong Q_{16}$. Then \autoref{lemPhix} implies
\[d_{\lambda_1,\phi_i}+\ldots+d_{\lambda_4,\phi_i}+\mu_1 d_{\psi_1,\phi_i}+\mu_2 d_{\psi_2,\phi_i}+\mu_3 d_{\psi_3,\phi_i}\ge 0\]
for $i=1,2,3$. The shape of the decomposition matrix of $B$ yields $\mu_1=\mu_2=\mu_3=-1$ as claimed. 
For the purpose of constructing an infinite family of examples, let $q$ be an odd prime and $H:=\SL(2,q)$. Let $\zeta\in \FF_{q^2}^\times$ of order $2(q-1)$. Then 
\[\hat{H}:=H\Bigl\langle \begin{pmatrix}
\zeta&0\\0&\zeta^{-1}
\end{pmatrix}\Bigr\rangle\le\SL(2,q^2)\] 
is a non-split extension with Sylow $2$-subgroup $E\cong Q_{2^{d+1}}$. Thus, we can apply \autoref{examples} to the pair $(H,\hat{H})$. For $q=3$ we end up with the (unique) non-principal block of $G=\mathtt{SmallGroup}(144,124)$.

Now assume that $E\cong SD_{16}$. 
Here we need to investigate the generalized decomposition matrix $Q^z$ with respect to a $B$-subsection $(z,b_z)$ where $\Z(D)=\langle z\rangle$. The columns of $Q^z$ lie in the orthogonal complement of the $\ZZ$-module spanned by the columns of the ordinary decomposition matrix and the column $d_{.,\phi_x}^x$. It is easy to find a basis of this $\ZZ$-module. Therefore, there exists an integral matrix $S\in\GL(3,\QQ)$ such that
\[Q^z=\begin{pmatrix}
1&.&.\\
.&1&.\\
.&.&1\\
1&1&1\\
-1&-1&.\\
-1&.&-1\\
.&-1&-1
\end{pmatrix}S.\]
By \autoref{lemPhix}, $(\epsilon_1,\ldots,\epsilon_4,\mu_1,\mu_2,\mu_3)Q^z=(0,0,0)$. After we multiply both sides with $S^{-1}$, we get $\mu_1=\mu_2=\mu_3=1$ as desired. 
To construct examples, let $H:=\SL(2,q)$ for an odd prime $q$. Let $\FF_q^\times=\langle\zeta\rangle$. The conjugation with $\begin{sm}0&\zeta\\1&\zeta\end{sm}\in\GL(2,q)$ induces an automorphism $\alpha$ on $H$ of order $2$. Then $\hat{H}:=H\rtimes\langle\alpha\rangle$ has Sylow $2$-subgroup $E\cong SD_{2^{d+1}}$, so we can apply \autoref{examples}.
For $q=3$, this gives the non-principal block of $G=\mathtt{SmallGroup}(144,125)$. 

\textbf{Case~3:} $B$ is Morita equivalent to the principal block of $\SL(2,5)$.\\
Here we have $(\epsilon_1,\ldots,\epsilon_4)=(1,0,0,1)$ with the labeling of the proof of \autoref{Q8h0}. The decomposition matrix shows further that $\phi_2=\overline{\phi_3}$ and therefore $\mu_1=\mu_2=0$. As before we have
\[(1+\mu_3)\bigl(2\phi_1(1)+\phi_2(1)+\phi_3(1)\bigr)=\lambda_1(1)+\lambda_4(1)+\mu_3\psi_3(1)\ge 0\]
with equality if and only if $E\cong Q_{16}$. Hence, $\mu_3=-1$ if $E\cong Q_{16}$ and $\mu_3=1$ if $E\cong SD_{16}$.
The former case occurs for a non-principal block of $G=\mathtt{SmallGroup}(720,414)$ and the latter for a non-principal block of $G=\mathtt{SmallGroup}(720,415)$ (same construction as in Case~2 with $q=5$).
\end{proof}

If $E\notin\{Q_{16},SD_{16}\}$, then one can show that $E\in\{D,D\times C_2,D*C_4\}$. 
It is possible to obtain some further information in these cases, but ultimately we do not know the F-S indicator $\mu$ of the unique (real) character of height $1$ when $l(B)=1$ and $E\cong D\times C_2$. \autoref{conC} would imply that $\mu=-1$.

\section{Homocyclic defect groups}

Since tame blocks have metacyclic defect groups, it is reasonable to look at other classes of $2$-blocks with metacyclic defect groups. The corresponding Morita equivalence classes have been determined in \cite[Theorem~1.1]{EKKS} (combined with \cite[Theorem~8.1]{habil}). The only non-nilpotent wild blocks $B$ occur for $D\cong C_{2^n}^2$ where $n\ge 2$. In this case $B$ is Morita equivalent to $F[D\rtimes C_3]$. In particular, the Morita equivalence class is uniquely determined by $l(B)$. We determine the F-S indicators in the special case $n=2$. Again these numbers only depend on the extended defect group. 

\begin{Thm}\label{homocyc}
Let $B$ be a real $2$-block with defect pair $(D,E)$ such that $D\cong C_4^2$, $E=D$ or $E\cong\mathtt{SmallGroup}(32,\textup{id})$. Then $(k(B),l(B))\in\{(16,1),(8,3)\}$ and exactly four characters in $\Irr(B)$ are $2$-rational. The F-S indicators of characters in $\Irr(B)$ are given below, where the first four characters are $2$-rational. If $l(B)=3$, then the first three characters are irreducible modulo $2$. All cases occur.
\begin{center}
\begin{tabular}{lll}
$l(B)$&\textup{id}&F-S indicators\\\hline
$1$&$D$, $21$ $(D\times C_2)$, $24$, $33$&$1, 1, 1, 1; 0^{12}$\\
&$3$ $(C_8\times C_4)$, $4$ $(C_8\rtimes C_4)$&$1, 1, -1, -1; 0^{12}$\\
&$25$ $(D_8\times C_4)$, $31$&$1, 1, 1, 1; 1, 1, 1, 1, 0^8$\\
&$26$ $(Q_8\times C_4)$, $32$&$1, 1, 1, 1; -1, -1, -1, -1, 0^8$\\
&$12$ $(C_4\rtimes C_8)$&$1, 1, -1, -1; 1, 1,-1, -1,0^8$\\
&$35$ $(C_4\rtimes Q_8)$&$1, 1, 1, 1; 1, 1, 1, 1, (-1)^8$\\
&$11$ $(C_4\wr C_2)$&$1, 1, 0, 0; 1, 1, 0^{10} $\\
&$34$&$1^{16}$\\\hline
$3$&$D$, $21$ $(D\times C_2)$, $33$&$1,0,0,1;0,0,0,0$\\
&$11$ $(C_4\wr C_2)$&$1,1,1,1;0,0,1,1$\\
&$34$&$1,0,0,1;1,1,1,1$
\end{tabular}
\end{center}
\end{Thm}
\begin{proof}
Let $l(B)=1$. Then $B$ is nilpotent and the generalized decomposition matrix of $B$ coincides with the character table of $D$. This shows that $k(B)=16$ and exactly four characters are $2$-rational. The possible groups $E$ can be computed with GAP and examples can be found among the groups of order $96$ as in \autoref{examples}. If $B$ is the principal block (i.\,e. $E=D$), then  $\Irr(B)=\Irr(G/\pcore_{2'}(G))=\Irr(D)$. In this case the claim is easy to check. Otherwise, the F-S indicators are determined by \cite[Theorem~10]{SambaleReal}. We note that the embedding of $D$ in $E$ is not always unique, but the F-S indicators are independent of this embedding (in our situation). 

Now let $l(B)=3$. Suppose first that $B$ is the principal block. By a result of Brauer~\cite[Theorem~1]{BrauerApp2}, we have $\Irr(B)=\Irr(G/\pcore_{2'}(G))=\Irr(D\rtimes C_3)$. The F-S indicators can be computed easily here. Now let $E\ne D$.
We argue as in \autoref{nonexist} to exclude most candidates for $E$. Let $(D,b_D)$ be a fixed Brauer pair. By \cite[Proposition~8(i)]{SambaleReal}, the extended stabilizer has the form $\N_G(D,b_D)^*=E\N_G(D,b_D)$. It can be checked that $\N_G(D,b_D)/\C_G(D)$ is isomorphic to a Sylow $3$-subgroup $S$ of $\Aut(D)$. Moreover, the normalizer of $S$ in $\Aut(S)$ has three conjugacy classes of involutions. Hence, there are at most four possible actions of $E$ on $D$ (including the trivial action). This excludes the cases $\mathrm{id}\in\{4,12,24,25,26,31,32\}$. 
Now let $\mathrm{id}=3$, i.\,e. $E\cong C_8\times C_4$. Then $b_D$ is real and nilpotent. By the first part of the proof $b_D$ has exactly $12$ non-real characters. Under the action of $\N_G(D,b_D)$ the $16$ characters in $\Irr(b_D)$ distribute into five orbits of length $3$ and one orbit of length $1$. In particular, the number of non-real characters in $\Irr(b_D)$ cannot be $12$. Contradiction. 

Next assume that $\mathrm{id}=35$. Here, $E$ acts on $D$ by inverting its elements. In particular, $E$ centralizes $D_1:=\langle x^2:x\in D\rangle\cong C_2^2$. Fix a $B$-subpair $(D_1,b_1)$. Then $b_1$ has defect pair $(\C_D(D_1),\C_E(D_1))=(D,E)$. In particular, $b_1$ is real. Since $S$ does not centralize $D_1$, $b_1$ is nilpotent. By the first part of the proof, $b_1$ has exactly eight characters with F-S indicator $-1$. But this leads to a contradiction by considering the action of $\N_G(D,b_D)$ in $\Irr(b_1)$ as above.
This leaves the cases $\mathrm{id}\in\{11,21,33,34\}$. In all of those, $E$ splits over $D$. Hence, all F-S indicators are non-negative by \cite[Theorem~5.6]{Gow}.

By \cite{EKKS}, the decomposition matrix of $B$ is 
\[Q:=\begin{pmatrix}
1&0&0\\
0&1&0\\
0&0&1\\
1&1&1\\
1&1&1\\
1&1&1\\
1&1&1\\
1&1&1
\end{pmatrix}.\]
Up to conjugation there are five non-trivial $B$-subsections $(x,b_x)$, $(y,b_y)$, $(y^{-1},b_y)$, $(z,b_z)$ and $(z^{-1},b_z)$ where $x$ is an involution and $y,z$ have order $4$. Thus, the complex conjugation fixes exactly four columns of the generalized decomposition matrix $\hat{Q}$. By Brauer's permutation lemma, there are exactly four $2$-rational characters. We stress that the $2$-conjugate characters can still be real since $y$ might be conjugate to $y^{-1}$ via an element not fixing $b_y$. By the shape of $Q$ we may assume that the first four characters are $2$-rational. By \cite[Theorem~5.3]{Gow}, two or four of them have F-S indicator $1$. 

By \cite[Theorem~15]{SambaleC4}, $B$ is isotypic to the principal block of $H:=D\rtimes C_3$. This implies via \cite[Proposition~7.3]{SambaleIso} that $\hat{Q}$ coincides up to basic sets with the generalized decomposition matrix of $H$. However, since $l(b_x)=l(b_y)=l(b_z)=1$, the columns of $\hat{Q}$ corresponding to $x,y,z$ are uniquely determined up to signs. Using the orthogonality relations the signs can be chosen such that
\[\hat{Q}=\begin{pmatrix}
1&0&0&\epsilon_x&\epsilon_y&\epsilon_y&\epsilon_z&\epsilon_z\\
0&1&0&\epsilon_x&\epsilon_y&\epsilon_y&\epsilon_z&\epsilon_z\\
0&0&1&\epsilon_x&\epsilon_y&\epsilon_y&\epsilon_z&\epsilon_z\\
1&1&1&3\epsilon_x&-\epsilon_y&-\epsilon_y&-\epsilon_z&-\epsilon_z\\
1&1&1&-\epsilon_x&(-1+2i)\epsilon_y&(-1-2i)\epsilon_y&\epsilon_z&\epsilon_z\\
1&1&1&-\epsilon_x&(-1-2i)\epsilon_y&(-1+2i)\epsilon_z&\epsilon_z&\epsilon_z\\
1&1&1&-\epsilon_x&\epsilon_y&\epsilon_y&(-1+2i)\epsilon_z&(-1-2i)\epsilon_z\\
1&1&1&-\epsilon_x&\epsilon_y&\epsilon_y&(-1-2i)\epsilon_z&(-1+2i)\epsilon_z
\end{pmatrix}\]
where $\epsilon_x,\epsilon_y,\epsilon_z\in\{\pm1\}$ and $i=\sqrt{-1}$. Since $b_x$, $b_y$ and $b_z$ are nilpotent with abelian defect group, we can apply \autoref{locnil} for each $E$. This gives a linear system on the vector of F-S indicators. We have checked by computer that there is a unique solution (up to permuting the first three characters). 
Examples are given by $G=\mathtt{SmallGroup}(288,a)$ where $a=67,407,406,405$ for $\mathrm{id}=11,21,33,34$ respectively.
\end{proof}

It might be possible to conduct a similar analysis for defect groups $D\cong C_{2^n}^2$ with arbitrary $n\ge 2$. However, for $n=3$ there are already $27$ extended defect groups to consider.

\section*{Acknowledgment}
I thank John Murray for some helpful discussions on the open cases in \autoref{secQ}. Parts of this paper were written at the University of Valencia, from which I received great hospitality.
The work is supported by the German Research Foundation (\mbox{SA 2864/4-1}).

\end{document}